\documentclass[12pt,reqno,tbtags]{amsart}
\usepackage[utf8]{inputenc}
\usepackage[T1]{fontenc}

\usepackage{amssymb}
\usepackage{amsthm}
\usepackage{amsmath, bm}
\usepackage{bbm}
\usepackage{mathrsfs}
\usepackage{txfonts}
\usepackage{graphicx}
\usepackage{geometry}
\usepackage{cite} 
\usepackage{float}
\usepackage{hyperref}

\usepackage{esint}
\usepackage{pgfplots}
\pgfplotsset{compat=1.18}
\usepackage{tikz-cd}
\usepackage{mdframed}
\usetikzlibrary{intersections}

\newtheorem{theorem}{Theorem}[section]
\newtheorem{proposition}{Proposition}[section]

\newtheorem{example}{Example}[section]
\newtheorem{question}{Question}[section]     
\newtheorem{lemma}{Lemma}[section]
\newtheorem{remark}{Remark}[section]

\allowdisplaybreaks[4]

\def\R{\mathbb{R}}

\def\L{\mathcal{L}}
\def\Z{\mathbb{Z}}
\def\d{\mathrm{d}}

\newcommand\norm[1]{\|#1\|}
\newcommand\inpro[2]{\langle #1,#2\rangle}
\newcommand\Inpro[3]{\langle #1,#2\rangle_{#3}}

\newcommand\1{\mathbbm{1}}
\newcommand\Norm[2]{\left\|#1\right\|_{#2}}

\newcommand{\eps}{\varepsilon}
\def\d{\mathrm{d}}

\begin{document}

\title[Sets with no Riesz bases of exponentials]{Sets with no Riesz bases of exponentials}

\author{Zhiqiang Wan}
\thanks{}
\address{School of Mathematical Sciences, University of Science and Technology of China, No. 96 Jinzhai Road, Baohe District, Hefei, Anhui Province, China}
\email{ZhiQiang\_Wan576@mail.ustc.edu.cn}

\date{\today}

\begin{abstract}
We prove that sets in a certain class do not admit Riesz bases of exponentials.
In particular, this class contains disks and triangles in the plane.
\end{abstract}

\maketitle

\section{Introduction}
Representing functions by simple oscillations is a basic idea in analysis and its
applications. It is natural to ask that such representations be unique and stable.
Classical Fourier series have these properties, but an
orthogonal expansion in exponentials is not available on every domain. Riesz bases
provide a natural way to preserve these properties without requiring orthogonality.

The orthogonal case is closely related to Fuglede's spectral set conjecture.
A bounded measurable set $S\subset\R^d$ of positive measure is called
\emph{spectral} if $L^2(S)$ admits an orthogonal basis of exponentials.
Fuglede \cite{F74} conjectured that $S$ is spectral if and only if it tiles
$\R^d$ by translations, up to a set of measure zero.
The conjecture holds for unions of two intervals by
{\fontencoding{T1}\selectfont\symbol{138}}aba \cite{La01}, and for convex bodies
in all dimensions by Lev and Matolcsi \cite{LM22}.
For general sets, Tao \cite{T04} constructed spectral sets which do not tile
by translations in every dimension $d\ge5$.
More recently, Greenfeld and Kolountzakis \cite{GK26} showed that both directions
of the conjecture fail even for connected sets in sufficiently high dimensions.

The proof for convex bodies uses a necessary condition called weak tiling.
In this condition, the indicator of the complement equals an integral of
translated indicators of the set, almost everywhere. The integral is taken
against a positive, locally finite Borel measure \cite{LM22}.
Kolountzakis, Lev and Matolcsi developed this method further, with applications
to nonconvex polytopes, product sets and Cantor sets of positive measure
\cite{KLM23}. They also obtained geometric restrictions on convex polytopes
and on the lengths of the gaps between intervals \cite{KLM25}.
The condition is not sufficient for spectrality. Kiss, Londner, Matolcsi and
Somlai \cite{KLMS26} constructed a set which weakly tiles its complement
but is neither spectral nor a translational tile.

We now turn to exponential Riesz bases. Let us first recall that a complete system
$\{u_n\}\subset H$ in a separable Hilbert space is a \emph{Riesz basis}
if there exists $K\ge1$ such that
\begin{equation}\label{eq:riesz-basis}
    \frac{1}{K}\sum_n|c_n|^2
    \le \Norm{\sum_n c_nu_n}{H}^{2}
    \le K\sum_n|c_n|^2
\end{equation}
for every finitely supported scalar sequence $\{c_n\}$.
For a measurable set $S\subset\R^d$ with $0<|S|<\infty$, we call
$E(\Lambda):=\{e^{2\pi i\inpro{\lambda}{x}}:\lambda\in\Lambda\}$,
where $\Lambda\subset\R^d$ is countable, an \emph{exponential Riesz basis}
if it is a Riesz basis in $L^2(S)$.

Classical estimates and stability results for nonharmonic Fourier series
can be found in \cite{I36,K64}. For disconnected sets, Kozma and Nitzan proved that finite
unions of bounded intervals \cite{KN15}, and of bounded axis-parallel boxes
\cite{KN16}, admit exponential Riesz bases.
Debernardi and Lev \cite{DL22} proved existence for centrally symmetric convex
polytopes with centrally symmetric faces of all dimensions.
For recent constructions and quantitative results, we refer to
\cite{AD26,BM24,L24,PRW24}.

Kozma, Nitzan and Olevskii \cite{KNO23} constructed a bounded set of positive measure
in $\R$ with no exponential Riesz basis, using translations and energy estimates.
Motivated by this result, we study sets with no Riesz bases of exponentials.
In fact, even for basic planar domains,
the existence of exponential Riesz bases is a
well-known open problem. Two prominent examples are the disk and the triangle,
discussed in \cite{KNO23}. We state these questions as follows.

\begin{question}\label{ques:disk-triangle}
Let $D\subset\R^2$ be a disk of positive radius, and let $\Delta\subset\R^2$
be a nondegenerate triangle.
\begin{enumerate}
    \item[(i)] Does $L^2(D)$ admit a Riesz basis of exponentials?
    \item[(ii)] Does $L^2(\Delta)$ admit a Riesz basis of exponentials?
\end{enumerate}
\end{question}

In this paper, we give a negative answer to both parts of
Question~\ref{ques:disk-triangle}. More generally, we prove the following theorem.
We write $B(c,R):=\{x\in\R^d:|x-c|<R\}$, and identify sets which differ by a set
of Lebesgue measure zero.

\begin{theorem}\label{thm:main}
Let $d\ge2$, and let $S\subset\R^d$ be one of the following sets:
\begin{enumerate}
    \item[(i)] A Euclidean ball $B(c,R)$, where $R>0$.

    \item[(ii)] A bounded convex polytope with nonempty interior having a facet
    whose supporting hyperplane is parallel to that of no other facet.
    In particular, this includes every nondegenerate $d$-simplex.

    \item[(iii)] A spherical shell $\{x\in\R^d:r<|x-c|<R\}$,
    where $0<r<R$.

    \item[(iv)] A finite union of Euclidean balls,
    \begin{equation*}
        S=\bigcup_{j=1}^{m}B(c_j,R_j),
        \qquad m\ge1,\quad R_j>0.
    \end{equation*}
    The balls are allowed to overlap.

    \item[(v)] A finite union $S=\bigcup_{j=1}^{m}S_j$ of Euclidean balls or
    spherical shells, where
    \begin{equation*}
        S_j=\{x\in\R^d:a_j<|x-c_j|<b_j\},
        \qquad 0\le a_j<b_j,
    \end{equation*}
    and
    \begin{equation*}
        \operatorname{dist}(\overline{S_i},\overline{S_j})>0
        \qquad (i\ne j).
    \end{equation*}
    When $a_j=0$, $S_j$ is understood to be the ball $B(c_j,b_j)$.
\end{enumerate}
Then $L^2(S)$ admits no Riesz basis of exponentials. The same conclusion holds
for every set $LS+v$, where $L$ is an invertible real $d\times d$ matrix and
$v\in\R^d$. It also holds for $S\times K$ and its invertible affine images,
where $k\ge1$ and $K\subset\R^k$ is any Lebesgue measurable set with
$0<|K|<\infty$.
\end{theorem}

For $d=2$, part~(ii) includes every nondegenerate triangle and every convex
polygon with an odd number of edges, where edges are maximal boundary segments.
Part~(v) includes finite unions of concentric spherical shells, after merging
overlapping or touching radial intervals and discarding the boundary spheres.
The affine extension gives ellipsoids and the regions between concentric
homothetic ellipsoids. For a union, the same affine map is applied to every component.
The product conclusion includes cylinders and prisms over the sets listed above.
The factor $K$ need not be bounded or have nonempty interior.

\begin{remark}\label{rem:more-general-sets}
In fact, our method also yields the nonexistence of Riesz bases of exponentials
for many more general sets.
\end{remark}

Our argument is inspired by Kozma, Nitzan and Olevskii \cite{KNO23}.
Assuming that the area-one disk admits an exponential Riesz basis with constant $K$,
they prove that $K\ge\sqrt{\frac{1+\sqrt{5}}{2}}>1$.
Their proof extends a localized function of $L^2$-norm one by its Riesz expansion.
Since translations preserve coefficient energy, the Riesz bounds force energy
into thin crescents when $K$ is close to $1$. Averaging over the directions
of translation then gives the stated lower bound by Fubini's theorem.

We use the translation principle with an expansion that changes with the
observation window. Our aim is to make the extension vanish on the overlap
of two windows, so that the lower Riesz bound forces energy into their
difference for any finite $K$. We first explain this choice for the disk
$S=D=B(0,r)$, with $r>0$ fixed. Set $\eps_0=r/10$ and fix
$0<\eps<\eps_0$. For $t\in[0,1]$, put
$u_t=(\cos(2\pi t),\sin(2\pi t))$ and define
\begin{equation*}
    \begin{aligned}
        V_{\eps,t}&:=D-(r-\eps)u_t,
        &W_{\eps,t}&:=D-(r+\eps)u_t,\\
        Q_\eps&:=B(0,\eps/4),
        &L_{\eps,t}&:=W_{\eps,t}\setminus V_{\eps,t}.
    \end{aligned}
\end{equation*}
Here $D-a:=\{x-a:x\in D\}$. The parameter $t$ moves both centers once around
the origin, while their separation is $2\eps$. All these sets lie in the fixed
disk $U:=B(0,3r)$. For every $x\in Q_\eps$ and every $t$, the inequalities
$|x+(r-\eps)u_t|<r-3\eps/4$ and $|x+(r+\eps)u_t|>r+3\eps/4$ give
$Q_\eps\subset V_{\eps,t}\setminus W_{\eps,t}$.
Thus the same small ball remains inside every first window and outside every
second window. The reverse difference $L_{\eps,t}$ is the opposite, moving
crescent. Each point belongs to it for a set of parameters of length at most
$C\sqrt{\eps/r}$, where $C$ is independent of $\eps$ and the point.
Throughout the parameter average, $\eps$ and $Q_\eps$ are fixed.
Figure~\ref{fig:disk-windows} shows two positions along this path.

Suppose that this fixed disk $D$ has an exponential Riesz basis with constant $K$.
Translation preserves this basis and its bounds. For fixed $\eps$, choose
one function $f_\eps$, zero outside $Q_\eps$, with $\Norm{f_\eps}{L^2(U)}=1$.
On each $V_{\eps,t}$, we expand this same function and extend its series to $U$.
Denote the extension by $F_{\eps,t}$.
It vanishes almost everywhere on $V_{\eps,t}\cap W_{\eps,t}$.
Thus all the energy required by the lower Riesz bound on $W_{\eps,t}$ lies
in $L_{\eps,t}$.
This gives a positive lower bound for every finite $K$, without an overlap error.

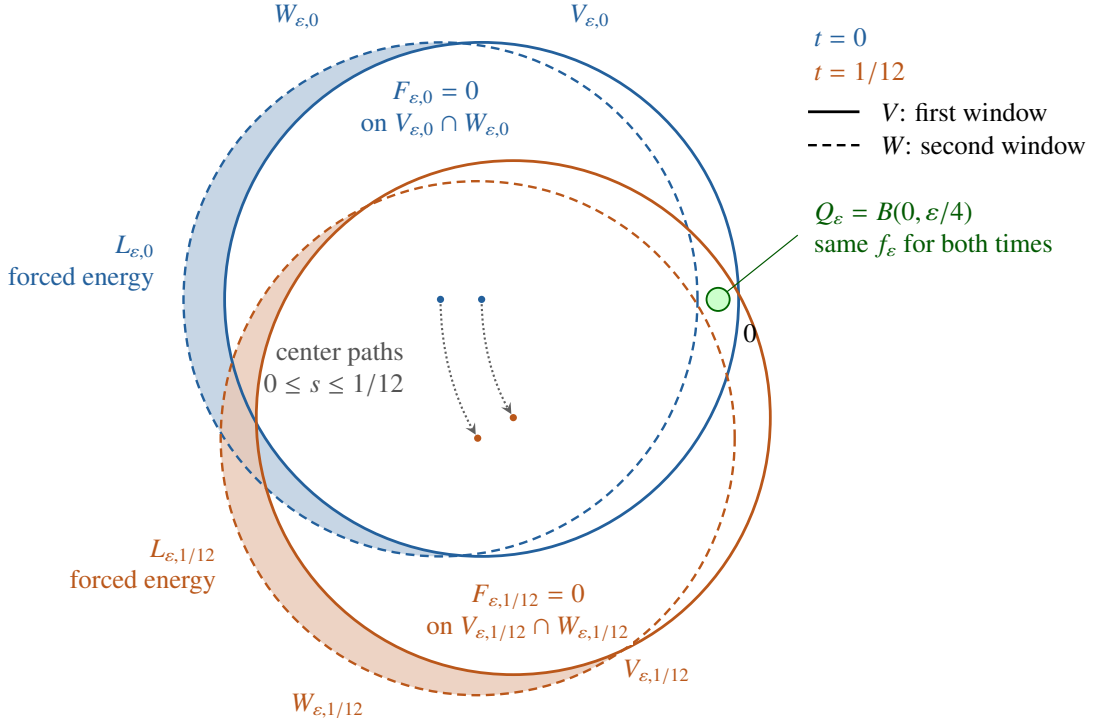
\begin{figure}[H]
    \centering
    \begin{tikzpicture}[scale=3.4, every node/.style={font=\footnotesize}]
        \definecolor{windowblue}{RGB}{35,96,156}
        \definecolor{windoworange}{RGB}{186,87,26}
        \coordinate (cwzero) at (-1.08,0);
        \coordinate (cvzero) at (-0.92,0);
        \coordinate (cwone) at ({-1.08*cos(30)},{-1.08*sin(30)});
        \coordinate (cvone) at ({-0.92*cos(30)},{-0.92*sin(30)});
        \foreach \suffix/\col in {zero/windowblue,one/windoworange} {
            \begin{scope}
                \clip (cw\suffix) circle (1);
                \fill[\col!32,opacity=0.8,even odd rule]
                    (-2.2,-1.7) rectangle (0.3,1.2)
                    (cv\suffix) circle (1);
            \end{scope}
        }
        \foreach \suffix/\col in {zero/windowblue,one/windoworange} {
            \draw[\col,line width=0.9pt,dash pattern=on 3.3pt off 2.1pt]
                (cw\suffix) circle (1);
            \draw[\col,line width=1.05pt] (cv\suffix) circle (1);
        }
        \draw[gray!75!black,line width=0.8pt,densely dotted,->,>=stealth,shorten >=1.5pt]
            (-1.08,0) arc[start angle=180,end angle=210,radius=1.08];
        \draw[gray!75!black,line width=0.8pt,densely dotted,->,>=stealth,shorten >=1.5pt]
            (-0.92,0) arc[start angle=180,end angle=210,radius=0.92];
        \foreach \suffix/\col in {zero/windowblue,one/windoworange} {
            \fill[\col] (cw\suffix) circle (0.014);
            \fill[\col] (cv\suffix) circle (0.014);
        }
        \node[anchor=east,align=right,text=gray!65!black] at (-1.18,-0.27)
            {center paths\\$0\le s\le1/12$};
        \filldraw[fill=green!20,draw=green!40!black,line width=0.7pt]
            (0,0) circle (0.045);
        \node[anchor=north west] at (0.055,-0.055) {$0$};
        \draw[green!35!black,line width=0.45pt] (0.035,0.030) -- (0.31,0.25);
        \node[anchor=west,align=left,green!35!black] at (0.33,0.27)
            {$Q_\eps=B(0,\eps/4)$\\same $f_\eps$ for both times};

        \node[windowblue,anchor=south] at (-1.64,1.01) {$W_{\eps,0}$};
        \node[windowblue,anchor=south] at (-0.50,1.01) {$V_{\eps,0}$};
        \node[windoworange,anchor=north] at (-1.52,-1.49) {$W_{\eps,1/12}$};
        \node[windoworange,anchor=north] at (-0.24,-1.36) {$V_{\eps,1/12}$};
        \node[windowblue,anchor=east,align=right] at (-2.15,0.14)
            {$L_{\eps,0}$\\forced energy};
        \node[windoworange,anchor=east,align=right] at (-1.91,-1.04)
            {$L_{\eps,1/12}$\\forced energy};
        \node[windowblue,align=center,fill=white,inner sep=1.5pt] at (-1.10,0.74)
            {$F_{\eps,0}=0$\\on $V_{\eps,0}\cap W_{\eps,0}$};
        \node[windoworange,align=center,fill=white,inner sep=1.5pt] at (-0.74,-1.22)
            {$F_{\eps,1/12}=0$\\on $V_{\eps,1/12}\cap W_{\eps,1/12}$};

        \node[anchor=west,windowblue] at (0.33,1.02) {$t=0$};
        \node[anchor=west,windoworange] at (0.33,0.88) {$t=1/12$};
        \draw[line width=1.05pt] (0.35,0.73) -- (0.55,0.73);
        \node[anchor=west] at (0.59,0.73) {$V$: first window};
        \draw[line width=0.9pt,dash pattern=on 3.3pt off 2.1pt]
            (0.35,0.60) -- (0.55,0.60);
        \node[anchor=west] at (0.59,0.60) {$W$: second window};
    \end{tikzpicture}
    \makeatletter
    \def\@captionfont{\normalfont\small}
    \makeatother
    \caption{Two configurations for the disk with $\eps/r=0.08$.
    Blue corresponds to $t=0$ and orange to $t=1/12$.
    The first windows $V_{\eps,t}$ have solid boundaries and the second
    windows $W_{\eps,t}$ have dashed boundaries. The dotted arrows follow
    the center paths from blue to orange. Both configurations share
    $Q_\eps=B(0,\eps/4)$ and the same function $f_\eps$.
    The green ball is enlarged for visibility.
    For each $t$, the extension $F_{\eps,t}$ vanishes on
    $V_{\eps,t}\cap W_{\eps,t}$, so its energy on $W_{\eps,t}$ lies in
    $L_{\eps,t}=W_{\eps,t}\setminus V_{\eps,t}$.
    All windows lie in $U=B(0,3r)$.}
    \label{fig:disk-windows}
\end{figure}

The remaining difficulty is that $F_{\eps,t}$ varies with $t$ and could
follow the moving loss region. A small average overlap of the regions alone
does not control this family. Along our chosen path, however, each point
enters or leaves the first window only a bounded number of times.
We use this property to approximate all the extensions, at any fixed
accuracy in $L^2(U)$, by finitely many functions which do not depend on $t$.
Their number is bounded independently of $\eps$.

For each approximating function, Fubini's theorem bounds its average loss
by its total energy times the parameter overlap bound.
This tends to zero. Since the number of approximating functions is uniformly
bounded, the same conclusion holds for the whole family.
The disk $D$, the assumed basis, its constant $K$, and $U$ are kept fixed.
First we let $\eps\downarrow0$ within $(0,r/10)$ at fixed approximation accuracy,
and then let that error tend to zero. The function $f_\eps$ may change with
$\eps$, but never with $t$. The average loss energy must therefore tend to zero.
This contradicts the positive lower bound from every second window:
\begin{equation*}
    K^{-2}\le
    \int_0^1\int_{L_{\eps,t}}|F_{\eps,t}(x)|^2\,\d x\,\d t
    \longrightarrow0\qquad{\rm{as}}\ \eps\to0.
\end{equation*}
For the other sets in Theorem~\ref{thm:main}, we choose different translation
paths with the same properties. The function has a common region of support.
Each point enters or leaves the first window a uniformly bounded number of
times. The loss regions have a uniformly small parameter overlap.

The rest of the paper is organized as follows.
In Section~\ref{sec:main-proof}, we give the abstract argument and prove the
nonexistence criterion.
In Section~\ref{sec:3}, we construct the windows for each class of sets in
Theorem~\ref{thm:main} and complete its proof.
Finally, in Section~\ref{sec:weaker-assumptions}, we replace the uniform
loss bound and the finite variation condition by weighted loss estimates and
mean square approximation. We also apply the resulting criterion to certain
countable unions of balls and give an example for which the chosen windows
have no uniform bound on their variation.

\section{Main proof}\label{sec:main-proof}

We first put the expansions on different translates into a common space.
This lets us compare their coefficients for a fixed function. Throughout this section, inner products are linear in the
first variable.

Suppose that $S\subset\R^d$ is bounded and measurable with nonempty interior, and that
$E(\Lambda)$ is a Riesz basis in $L^2(S)$ with bounds $0<A\le B<\infty$.
Write $e_\lambda(x)=e^{2\pi i\inpro{\lambda}{x}}$.
For $v\in\R^d$, the map
\begin{equation*}
    D_v(c_\lambda)_{\lambda\in\Lambda}
    =\bigl(e^{2\pi i\inpro{\lambda}{v}}c_\lambda\bigr)_{\lambda\in\Lambda}
\end{equation*}
is unitary on $\ell^2(\Lambda)$. The identity
\begin{equation*}
    \sum_{\lambda\in\Lambda}c_\lambda e_\lambda(x+v)
    =\sum_{\lambda\in\Lambda}(D_vc)_\lambda e_\lambda(x)
\end{equation*}
therefore shows that $E(\Lambda)$ is a Riesz basis on every translate $S+v$,
with the same bounds $A,B$.

Let $U$ be a fixed bounded measurable set containing all the translates used
in the argument. Since $S$ contains an open ball, finitely many translates
$S+v_1,\ldots,S+v_m$ cover $U$. For every finitely supported $c$, we have
\begin{equation*}
    \int_U\left|\sum_{\lambda\in\Lambda}c_\lambda e_\lambda(x)\right|^2\d x
    \le\sum_{j=1}^{m}\int_{S+v_j}
       \left|\sum_{\lambda\in\Lambda}c_\lambda e_\lambda(x)\right|^2\d x
    \le mB\norm{c}_{\ell^2(\Lambda)}^2.
\end{equation*}
Thus the finite sums extend to a bounded operator
\begin{equation}\label{eq:main-synthesis}
    T:\ell^2(\Lambda)\longrightarrow L^2(U),
    \qquad C:=\norm{T}^2\le mB.
\end{equation}
The expansion on each window is the restriction of this same operator.
In particular, $C$ stays fixed when the windows move inside $U$.

For a measurable set $E\subset U$, write
\begin{equation*}
    P_EF=\1_EF,\qquad T_Ec=(Tc)|_E,
    \qquad X_E=T_E^*T_E=T^*P_ET.
\end{equation*}
If $E$ is a translate of $S$, then $T_E$ is invertible and
$AI\le X_E\le BI$. Given $f\in L^2(E)$, let $\widetilde f$ denote its zero
extension to $U$. Its analysis vector and its expansion coefficients are
\begin{equation}\label{eq:main-analysis-coefficients}
    a=T_E^*f=T^*\widetilde f,
    \qquad c=T_E^{-1}f=X_E^{-1}a.
\end{equation}
Indeed, if $\delta_\lambda$ is the coordinate vector in $\ell^2(\Lambda)$,
the dual basis is $\psi_\lambda^E=(T_E^{-1})^*\delta_\lambda$
(see \cite{OU16}), and
\begin{equation*}
    f=\sum_{\lambda\in\Lambda}\Inpro{f}{\psi_\lambda^E}{L^2(E)}e_\lambda
     =\sum_{\lambda\in\Lambda}\Inpro{f}{e_\lambda}{L^2(E)}\psi_\lambda^E.
\end{equation*}
When a function is supported in every first window, its vector $T^*f$ is the
same for all these windows. Formula~\eqref{eq:main-analysis-coefficients}
then describes how its coefficients change with the window.

We now state the criterion in a form that only uses a bounded operator and
measurable windows. Let $H$ be a separable complex Hilbert space, let
$U\subset\R^d$ be measurable with finite measure, and let
$T:H\to L^2(U)$ be bounded. We retain the notation $C=\norm{T}^2$,
$P_E$, and $T_E$ introduced above. The parameter interval $I=[0,1]$ is
equipped with a Borel probability measure $\nu$.

\begin{theorem}\label{thm:abstract-windows}
There do not exist constants $\eps_0>0$, $0<A\le B<\infty$ and $0\le J<\infty$,
and measurable sets $V_{\eps,t},W_{\eps,t}\subset U$ for
$0<\eps<\eps_0$ and $t\in I$, satisfying all the following conditions.
For each fixed $\eps$, the indicators of the two windows are jointly measurable
in $(t,x)$.
\begin{enumerate}
    \item[(i)] The operator $T_{V_{\eps,t}}$ is onto, and
    \begin{equation}\label{eq:main-window-bounds}
        \begin{aligned}
            AI\le X_{\eps,t}:=T^*P_{V_{\eps,t}}T\le BI,\\
            AI\le Y_{\eps,t}:=T^*P_{W_{\eps,t}}T\le BI
        \end{aligned}
    \end{equation}
    for every $\eps,t$.

    \item[(ii)] For each $\eps$ there is a measurable set $Q_\eps\subset U$
    of positive measure such that
    \begin{equation*}
        \bigl|Q_\eps\setminus(V_{\eps,t}\setminus W_{\eps,t})\bigr|=0
        \qquad(t\in I).
    \end{equation*}

    \item[(iii)] Writing $L_{\eps,t}=W_{\eps,t}\setminus V_{\eps,t}$, we have
    \begin{equation}\label{eq:main-loss-overlap}
        \omega_\eps:=\operatorname*{ess\,sup}_{x\in U}
        \int_I\1_{L_{\eps,t}}(x)\,\d\nu(t)
        \longrightarrow0\qquad{\rm{as}}\ \eps\to0.
    \end{equation}

    \item[(iv)] For each $\eps$, the indicator representatives satisfy
    \begin{equation}\label{eq:main-turnover}
        \sum_{j=1}^{n}
        \left|\1_{V_{\eps,t_j}}(x)-\1_{V_{\eps,t_{j-1}}}(x)\right|
        \le J
    \end{equation}
    for almost every $x\in U$ and every finite sequence
    $t_0<\cdots<t_n$ in $I$. The exceptional null set is independent of
    the chosen sequence.
\end{enumerate}
\end{theorem}

The first two conditions allow us to reconstruct a common function with
$L^2(U)$-norm one and force positive energy into each $L_{\eps,t}$.
We use the notation for the disk in the Introduction to explain the idea.
Fix $\eps$ and choose a function $f_\eps$ supported in $Q_\eps$,
with $\Norm{f_\eps}{L^2(U)}=1$. Let $c_{\eps,t}$ be its expansion coefficients
on $V_{\eps,t}$ and write $F_{\eps,t}=Tc_{\eps,t}$.
Since $F_{\eps,t}=f_\eps$ on $V_{\eps,t}$ and $Q_\eps$ is disjoint from
$W_{\eps,t}$, the extension vanishes almost everywhere on
$V_{\eps,t}\cap W_{\eps,t}$.
The upper Riesz bound on $V_{\eps,t}$ gives $\norm{c_{\eps,t}}^2\ge1/B$,
so the lower bound on $W_{\eps,t}$ yields
\begin{equation*}
    \int_{L_{\eps,t}}|F_{\eps,t}(x)|^2\,\d x
    =\int_{W_{\eps,t}}|F_{\eps,t}(x)|^2\,\d x
    \ge A\norm{c_{\eps,t}}^2\ge\frac AB.
\end{equation*}

To understand condition~(iii), we consider the following toy model.
Let $G\in L^2(U)$ be fixed as $t$ varies. By Fubini's theorem,
\begin{equation*}
    \begin{aligned}
        \int_I\int_{L_{\eps,t}}|G(x)|^2\,\d x\,\d\nu(t)
        &=\int_U|G(x)|^2
          \left(\int_I\1_{L_{\eps,t}}(x)\,\d\nu(t)\right)\d x\\
        &\le\omega_\eps\Norm{G}{L^2(U)}^2.
    \end{aligned}
\end{equation*}
The last inequality follows directly from the definition of $\omega_\eps$.
By condition~(iii), the average energy tends to zero uniformly over functions
with bounded $L^2(U)$-norm. The functions $F_{\eps,t}$ vary with $t$, so their energy could
move with $L_{\eps,t}$.

A simple example shows the problem. On the circle $\R/\Z$, take $0<\eps<1$
and let $J_{\eps,t}=[t,t+\eps)$ modulo $1$, with $t\in[0,1]$.
Set $G_{\eps,t}=\eps^{-1/2}\1_{J_{\eps,t}}$.
Each point belongs to $J_{\eps,t}$ for a set of times of Lebesgue measure $\eps$, but
\begin{equation*}
    \int_{J_{\eps,t}}|G_{\eps,t}(x)|^2\,\d x
    =\Norm{G_{\eps,t}}{L^2(\R/\Z)}^2=1.
\end{equation*}
The function moves with the interval and keeps all its energy there, so the
average energy stays equal to one as $\eps\to0$.
Condition~(iv) and the Riesz bounds let us approximate $F_{\eps,t}$ in $L^2(U)$
by finitely many functions fixed in $t$. At each fixed accuracy, these functions
may depend on $\eps$, but their number has a bound independent of $\eps$.

For the disk, condition~(iv) follows directly from the motion of the first window.
Recall that $V_{\eps,t}=B(-(r-\eps)u_t,r)$.
For fixed $x\ne0$, the condition $x\in V_{\eps,t}$ is equivalent to
\begin{equation*}
    x\cdot u_t<\frac{r^2-|x|^2-(r-\eps)^2}{2(r-\eps)}.
\end{equation*}
The left-hand side is a cosine function of $2\pi t$.
During one full turn, the allowed directions form an arc, possibly empty or
the whole circle. Thus the indicator changes value at most twice.
At $x=0$, it is constant. Hence condition~(iv) holds with $J=2$,
independently of $\eps$.

If the disk admitted an exponential Riesz basis, the windows above would
satisfy all four conditions of Theorem~\ref{thm:abstract-windows}.
We will give all the details in Subsection~\ref{subsec:balls-shells}.
Taking the theorem as given for now,
we therefore obtain a contradiction and rule out such a basis.
We next prove the lemmas needed for the abstract argument, which also applies
to more general sets. In each case, the aim is to force a fixed positive lower bound
on the average loss energy and then show that this energy tends to zero.

We call a finite set of centres in $H$ a \emph{finite $\eta$-net} of a given set
if every point of that set is at distance at most $\eta$ from one of the centres.
A subset of a Hilbert space is precompact if and only if it admits
a finite $\eta$-net for every $\eta>0$ (see \cite{Lax02}).

\begin{lemma}\label{lem:main-direct-cover}
Suppose that the windows satisfy \eqref{eq:main-turnover}. For every $v\in H$
and $\eta>0$, the set $\{X_{\eps,t}v:t\in I\}$ has an $\eta$-net whose
centres belong to this set and whose cardinality is at most
\begin{equation}\label{eq:main-direct-number}
    \left\lceil1+\frac{JC^2\norm{v}^2}{\eta^2}\right\rceil.
\end{equation}
The bound is independent of $\eps$.
\end{lemma}
\begin{proof}
Fix $\eps$ and omit it from the notation. For any $t_0<\cdots<t_n$, we have
\begin{equation}\label{eq:main-square-variation}
    \begin{aligned}
    \sum_{j=1}^{n}\norm{(X_{t_j}-X_{t_{j-1}})v}^2
    &\le C\sum_{j=1}^{n}
       \Norm{(P_{V_{t_j}}-P_{V_{t_{j-1}}})Tv}{L^2(U)}^2\\
    &=C\int_U|Tv(x)|^2\sum_{j=1}^{n}
       \left|\1_{V_{t_j}}(x)-\1_{V_{t_{j-1}}}(x)\right|^2\d x\\
    &\le JC\Norm{Tv}{L^2(U)}^2
     \le JC^2\norm{v}^2.
    \end{aligned}
\end{equation}
Here the square of each indicator difference equals its absolute value.

If $m$ orbit points have pairwise distances greater than $\eta$, order their
parameters increasingly. Consecutive points still have distance greater than
$\eta$, so \eqref{eq:main-square-variation} gives
$(m-1)\eta^2\le JC^2\norm{v}^2$.
Starting with any orbit point, keep adding a point at distance greater than
$\eta$ from all previously chosen points whenever one exists.
The procedure terminates after finitely many steps, with the number of chosen
points bounded by \eqref{eq:main-direct-number}. The resulting set is the required net.
\end{proof}

Lemma~\ref{lem:main-direct-cover} shows that the variation bound on the windows
gives a finite approximation to the family $\{X_{\eps,t}v:t\in I\}$ for each
fixed $v$. At any fixed accuracy, the number of centres is bounded independently
of $\eps$ and uniformly for vectors $v$ of bounded norm.

The coefficients in \eqref{eq:main-analysis-coefficients} involve $X_{\eps,t}^{-1}$.
We therefore need the same kind of estimate for the inverse operators.
The lower bound $A$ gives an iteration that reduces the error by a fixed
factor less than one. We apply the preceding lemma at each step.

\begin{lemma}\label{lem:main-inverse-cover}
Suppose that $AI\le X_{\eps,t}\le BI$ and that
\eqref{eq:main-square-variation} holds. For every $M>0$ and $\delta>0$,
there is a finite number $N_\delta$, depending only on $A,B,C,J,M,\delta$,
with the following property. For every $a\in H$ with $\norm{a}\le M$ and
every $\eps$, the set
\begin{equation*}
    \{X_{\eps,t}^{-1}a:t\in I\}
\end{equation*}
has a $\delta$-net of at most $N_\delta$ centres, each of norm at most $M/A$.
\end{lemma}
\begin{proof}
Fix $\eps,a$ and omit $\eps$ from the notation. Set
\begin{equation*}
    q=1-\frac AB,\qquad R_t=I-\frac{X_t}{B},
    \qquad \rho=\frac MA.
\end{equation*}
Then $0\le R_t\le qI$ and $0\le q<1$. Put $c_t=X_t^{-1}a$ and define
\begin{equation}\label{eq:main-iteration}
    y_{0,t}=0,\qquad y_{n+1,t}=\frac aB+R_ty_{n,t}.
\end{equation}
Since $c_t=a/B+R_tc_t$ and $M/B+q\rho=\rho$, induction gives
\begin{equation}\label{eq:main-iteration-error}
    \norm{y_{n,t}}\le\rho,
    \qquad\norm{c_t-y_{n,t}}\le\rho q^n.
\end{equation}
We next construct finite approximations to these iterates with an error bound
independent of $n$ and $\eps$.

Choose
\begin{equation}\label{eq:main-net-precision}
    h=\frac{(1-q)\delta}{2}=\frac{A\delta}{2B},
    \qquad
    L_\delta=\left\lceil1+\frac{4JC^2M^2}{A^4\delta^2}\right\rceil.
\end{equation}
For each fixed $z$ with $\norm{z}\le\rho$, the identity
$(R_t-R_s)z=-(X_t-X_s)z/B$ and the packing argument in the proof of
Lemma~\ref{lem:main-direct-cover} show that
$\{R_tz:t\in I\}$ has an $h$-net $\mathcal F_z$ with centres in that orbit and
\begin{equation*}
    \#\mathcal F_z
    \le\left\lceil1+\frac{JC^2\rho^2}{B^2h^2}\right\rceil
    =L_\delta.
\end{equation*}
Starting from $\mathcal E_0=\{0\}$, define
\begin{equation*}
    \mathcal E_{n+1}
    =\bigcup_{z\in\mathcal E_n}\left(\frac aB+\mathcal F_z\right).
\end{equation*}
Each new centre has norm at most $M/B+q\rho=\rho$. Thus,
\begin{equation}\label{eq:main-net-recursion}
    \#\mathcal E_n\le L_\delta^n,
    \qquad\norm{z}\le\rho\quad(z\in\mathcal E_n).
\end{equation}
Set $e_0=0$ and $e_{n+1}=qe_n+h$. We prove by induction that $\mathcal E_n$
covers all $y_{n,t}$ with error $e_n$. Suppose this holds at step $n$.
Choose $z\in\mathcal E_n$ within $e_n$ of $y_{n,t}$, and then choose a point
of $\mathcal F_z$ within $h$ of $R_tz$. Adding $a/B$ to the latter point gives
a point of $\mathcal E_{n+1}$ within $qe_n+h=e_{n+1}$ of $y_{n+1,t}$.
The induction starts at $\mathcal E_0=\{0\}$, and
\begin{equation*}
    e_n=h\sum_{j=0}^{n-1}q^j\le\frac\delta2.
\end{equation*}
Choose $m\ge1$ so that $\rho q^m\le\delta/2$, taking $m=1$ when $q=0$.
Equations~\eqref{eq:main-iteration-error} and \eqref{eq:main-net-recursion}
show that $\mathcal E_m$ is a $\delta$-net for all $c_t$ with
$\#\mathcal E_m\le N_\delta:=L_\delta^m$ and centre norms at most $\rho$.
The choices of $m$ and $L_\delta$ depend only on the stated constants.
\end{proof}

The centres may depend on $a$ and $\eps$, but at each fixed accuracy their
number and norms have uniform bounds. We can therefore apply the
small-overlap estimate to these finitely many vectors, which are fixed as
$t$ varies.

\begin{proof}[Proof of Theorem~\ref{thm:abstract-windows}]
Suppose that all four conditions hold. For each $\eps$, choose
\begin{equation*}
    f_\eps=|Q_\eps|^{-1/2}\1_{Q_\eps}\in L^2(U),
    \qquad a_\eps=T^*f_\eps.
\end{equation*}
The function $f_\eps$ has norm one and is fixed as $t$ varies.
By condition~(ii),
\begin{equation*}
    a_\eps=T_{V_{\eps,t}}^*(f_\eps|_{V_{\eps,t}}),
    \qquad\norm{a_\eps}\le\sqrt B.
\end{equation*}
Condition~(i) makes $T_{V_{\eps,t}}$ invertible. Hence
\begin{equation}\label{eq:main-exact-reconstruction}
    c_{\eps,t}:=X_{\eps,t}^{-1}a_\eps,
    \qquad T_{V_{\eps,t}}c_{\eps,t}=f_\eps|_{V_{\eps,t}}.
\end{equation}
Applying the two bounds on this window gives
\begin{equation}\label{eq:main-coefficient-size}
    \frac1B\le\norm{c_{\eps,t}}^2\le\frac1A.
\end{equation}

The function $Tc_{\eps,t}$ vanishes almost everywhere on
$V_{\eps,t}\cap W_{\eps,t}$. Thus all its energy on the second window lies in
$L_{\eps,t}$. More precisely, set
\begin{equation*}
    Z_{\eps,t}=T^*P_{L_{\eps,t}}T.
\end{equation*}
Then $0\le Z_{\eps,t}\le Y_{\eps,t}\le BI$, and
\begin{equation}\label{eq:main-forced-energy}
    \inpro{Z_{\eps,t}c_{\eps,t}}{c_{\eps,t}}
    =\Norm{T_{W_{\eps,t}}c_{\eps,t}}{L^2(W_{\eps,t})}^2
    \ge A\norm{c_{\eps,t}}^2\ge\frac AB.
\end{equation}

On the other hand, for every fixed $v\in H$, Tonelli's theorem gives
\begin{equation}\label{eq:main-fixed-energy}
    \begin{aligned}
    \int_I\inpro{Z_{\eps,t}v}{v}\,\d\nu(t)
    &=\int_U|Tv(x)|^2
      \left(\int_I\1_{L_{\eps,t}}(x)\,\d\nu(t)\right)\d x\\
    &\le\omega_\eps\Norm{Tv}{L^2(U)}^2
     \le C\omega_\eps\norm{v}^2.
    \end{aligned}
\end{equation}
To compare this with \eqref{eq:main-forced-energy}, fix $\delta>0$ and apply
Lemma~\ref{lem:main-inverse-cover} with $M=\sqrt B$. For each $\eps$, it gives
a finite set
\begin{equation*}
    \{v_{\eps,1},\ldots,v_{\eps,N_\eps}\},
    \qquad N_\eps\le N_\delta,
    \qquad\norm{v_{\eps,j}}\le\frac{\sqrt B}{A},
\end{equation*}
within distance $\delta$ of every $c_{\eps,t}$. The number $N_\delta$ is
independent of $\eps$.

For each $t$, choose a centre $v_{\eps,j}$ with
$\norm{c_{\eps,t}-v_{\eps,j}}\le\delta$.
Using $\norm{Z_{\eps,t}^{1/2}}\le\sqrt B$, we obtain
\begin{equation*}
    \frac AB
    \le\norm{Z_{\eps,t}^{1/2}c_{\eps,t}}^2
    \le2\norm{Z_{\eps,t}^{1/2}v_{\eps,j}}^2+2B\delta^2
    \le2\sum_{j=1}^{N_\eps}
       \inpro{Z_{\eps,t}v_{\eps,j}}{v_{\eps,j}}+2B\delta^2.
\end{equation*}
Each centre is fixed as $t$ varies, so the last expression is measurable.
Integrating and using \eqref{eq:main-fixed-energy} gives
\begin{equation}\label{eq:main-final-estimate}
    \frac AB
    \le\frac{2CB}{A^2}N_\delta\omega_\eps+2B\delta^2.
\end{equation}
Choose $\delta^2=A/(8B^2)$, so the second term is $A/(4B)$.
With this choice fixed, condition~(iii) allows us to take $\eps$ small enough
that the first term is less than $A/(4B)$.
This contradicts \eqref{eq:main-final-estimate}.
\end{proof}

For completeness, the joint measurability assumed above also makes
$t\mapsto X_{\eps,t}v$ strongly measurable for each fixed $v$.
The uniform expansion
\begin{equation*}
    X_{\eps,t}^{-1}=\frac1B\sum_{n=0}^{\infty}
    \left(I-\frac{X_{\eps,t}}B\right)^n
\end{equation*}
then gives strong measurability of $t\mapsto c_{\eps,t}$.
Thus the average energy used in the Introduction is well defined.

To apply the criterion to a bounded set $S$ with nonempty interior, assume
that $E(\Lambda)$ is a Riesz basis on $S$. If the first and second windows
are translates of $S$ contained in one fixed bounded set $U$, the construction
at the start of this section supplies $T$ and condition~(i).
It remains to choose the windows so that conditions~(ii)--(iv) hold.
Theorem~\ref{thm:abstract-windows} then gives the required contradiction.

\section{Geometric constructions}\label{sec:3}

For parts~(i)--(v) of Theorem~\ref{thm:main}, we use translates of $S$
contained in a fixed bounded set $U$.
Section~\ref{sec:main-proof} then gives condition~(i) of
Theorem~\ref{thm:abstract-windows}.
We verify the other conditions by finding a common set $Q_\eps$ of positive
measure, proving that $\omega_\eps\to0$, and bounding the variation of the
first-window indicators.

We use normalized Lebesgue measure on each parameter interval. An affine
change of parameter identifies it with $[0,1]$ and preserves the estimates
required by the abstract criterion. In parts~(i)--(v), the windows are continuous translates
of bounded Borel sets, so their indicators are jointly measurable.
A null-set modification does not change the Riesz basis property.
We represent each shell by the difference of its open outer and inner balls.

\subsection{A bound for moving balls}

Several cases use the same circular motion. Fix orthonormal vectors
$e_1,e_2\in\R^d$, put $E=\operatorname{span}\{e_1,e_2\}$, and write
\begin{equation*}
    u_\theta=e_1\cos\theta+e_2\sin\theta.
\end{equation*}
We first record an elementary estimate. For an interval $J_0\subset\R$
of length $h>0$,
\begin{equation}\label{eq:geom-cosine}
    \bigl|\{\theta\in[0,2\pi]:\cos\theta\in J_0\}\bigr|
    =2\int_{J_0\cap[-1,1]}\frac{\d s}{\sqrt{1-s^2}}
    \le4\sqrt{2h}.
\end{equation}
To see the last inequality, split the integral at zero. On $[0,1]$,
$(1-s^2)^{-1/2}\le(1-s)^{-1/2}$, whose integral over an interval of length
$h_+$ is at most $2\sqrt{h_+}$. The negative half has the same bound.
Adding the two contributions and using $h_++h_-\le h$ gives
\eqref{eq:geom-cosine}.

\begin{lemma}\label{lem:geom-moving-balls}
Let $R,a>0$ and $b\in\R^d$. Define
\begin{equation}\label{eq:geom-component-windows}
    V_{\eps,\theta}^{a}=B\bigl(b-(R-\eps)u_\theta,a\bigr),
    \qquad
    W_{\eps,\theta}^{a}=B\bigl(b-(R+\eps)u_\theta,a\bigr).
\end{equation}
There is a constant $C_{R,a}$, independent of $b,x,\eps$, such that
\begin{equation}\label{eq:geom-outer-loss}
    \sup_{x\in\R^d}
    \bigl|\{\theta\in[0,2\pi]:
        x\in W_{\eps,\theta}^{a}\setminus V_{\eps,\theta}^{a}\}\bigr|
    \le C_{R,a}\sqrt{\eps}
\end{equation}
whenever $0<\eps<R/2$. If $0<a<R$, then also
\begin{equation}\label{eq:geom-inner-loss}
    \sup_{x\in\R^d}
    \bigl|\{\theta\in[0,2\pi]:
        x\in V_{\eps,\theta}^{a}\setminus W_{\eps,\theta}^{a}\}\bigr|
    \le C_{R,a}\sqrt{\eps}
\end{equation}
for $0<\eps<\min\{R/2,(R-a)/2\}$. Each of the two ball indicators has total
variation at most $2$ on any interval of length at most $2\pi$.
\end{lemma}
\begin{proof}
Fix $x$, and set
\begin{equation*}
    z=x-b,\qquad \rho=|z|,\qquad p=|\operatorname{proj}_E z|,
    \qquad q_\theta=\inpro{z}{u_\theta}.
\end{equation*}
When $p>0$, there is $\phi\in\R$ such that
$q_\theta=p\cos(\theta-\phi)$. Put
\begin{equation*}
    H_\pm=\frac{a^2-\rho^2-(R\pm\eps)^2}{2(R\pm\eps)}.
\end{equation*}
Membership in $V_{\eps,\theta}^{a}$ and $W_{\eps,\theta}^{a}$ is given,
respectively, by $q_\theta<H_-$ and $q_\theta<H_+$.

Suppose that $x\in W_{\eps,\theta}^{a}\setminus V_{\eps,\theta}^{a}$
for some $\theta$. Subtracting the squared distances gives
\begin{equation*}
    4\eps(R+q_\theta)
    =|z+(R+\eps)u_\theta|^2-|z+(R-\eps)u_\theta|^2<0.
\end{equation*}
Thus $p>R$. Also $\rho<a+R+\eps$, and
\begin{equation}\label{eq:geom-threshold-width}
    0<H_+-H_-
    =\frac{\eps(\rho^2-a^2-R^2+\eps^2)}{R^2-\eps^2}
    \le\left(\frac{4a}{R}+2\right)\eps.
\end{equation}
The last bound uses $\eps<R/2$ and $\rho<a+3R/2$.
The loss parameters satisfy $H_-\le q_\theta<H_+$, so the corresponding
interval for $\cos(\theta-\phi)$ has length at most
$(4a/R+2)\eps/R$. Equation~\eqref{eq:geom-cosine} proves
\eqref{eq:geom-outer-loss}.

For the reverse difference, suppose that $a<R$ and
$x\in V_{\eps,\theta}^{a}\setminus W_{\eps,\theta}^{a}$.
Projection onto $E$ gives
\begin{equation*}
    p\ge R-\eps-a>\frac{R-a}{2}.
\end{equation*}
In this case $H_+\le q_\theta<H_-$ and
\begin{equation*}
    0<H_--H_+
    =\frac{\eps(a^2+R^2-\rho^2-\eps^2)}{R^2-\eps^2}
    \le\frac83\eps.
\end{equation*}
After division by $p$, the cosine interval has length at most
$16\eps/(3(R-a))$. Applying \eqref{eq:geom-cosine} proves
\eqref{eq:geom-inner-loss}. In both cases, a loss point must have $p>0$.

Finally, membership in either ball is one inequality in a cosine, or is
constant when $p=0$. Its indicator changes at most twice over one period.
The same variation bound holds on every shorter interval.
\end{proof}

The lower bounds for $p$ make these estimates uniform in the point $x$.
They will also allow us to use the same motion for several balls with
different centres and radii.

\subsection{Balls and spherical shells}\label{subsec:balls-shells}

First let $S=B(c,R)$. Fix $\eps_0=R/10$ and, for $0<\eps<\eps_0$, define
\begin{equation}\label{eq:geom-ball-windows}
    \begin{aligned}
        V_{\eps,\theta}&=S-c-(R-\eps)u_\theta,
        &W_{\eps,\theta}&=S-c-(R+\eps)u_\theta,\\
        Q_\eps&=B(0,\eps/4),
        &U&=B(0,3R).
    \end{aligned}
\end{equation}
Here $\theta\in[0,2\pi]$ and $\d\nu(\theta)=\d\theta/(2\pi)$.
For $x\in Q_\eps$,
\begin{equation*}
    |x+(R-\eps)u_\theta|<R-\frac{3\eps}{4},
    \qquad
    |x+(R+\eps)u_\theta|>R+\frac{3\eps}{4}.
\end{equation*}
Thus $Q_\eps\subset V_{\eps,\theta}\setminus W_{\eps,\theta}$.
All windows lie in $U$. Lemma~\ref{lem:geom-moving-balls}, with $a=R$ and
$b=0$, gives $\omega_\eps\le C_{R,R}\sqrt{\eps}/(2\pi)$ and the variation
bound $J=2$. Theorem~\ref{thm:abstract-windows} proves part~(i) of
Theorem~\ref{thm:main}.

Now let $S=B(c,R)\setminus B(c,a)$, where $0<a<R$. Use the same translates,
the same $Q_\eps$ and $U$, with
\begin{equation*}
    \eps_0=\min\left\{\frac R{10},\frac{R-a}{4}\right\}.
\end{equation*}
For $x\in Q_\eps$, the distance to the centre of the first window satisfies
\begin{equation*}
    a<R-\frac{5\eps}{4}
    <|x+(R-\eps)u_\theta|
    <R-\frac{3\eps}{4}<R.
\end{equation*}
The distance for the second window is greater than $R$, as before.
Thus the common-set condition still holds.

The loss can now arise either from the outer boundary or from the inner hole.
Using the notation \eqref{eq:geom-component-windows} with $b=0$,
\begin{equation}\label{eq:geom-shell-loss}
    W_{\eps,\theta}\setminus V_{\eps,\theta}
    \subset
    \bigl(W_{\eps,\theta}^{R}\setminus V_{\eps,\theta}^{R}\bigr)
    \cup
    \bigl(V_{\eps,\theta}^{a}\setminus W_{\eps,\theta}^{a}\bigr).
\end{equation}
The two parts of Lemma~\ref{lem:geom-moving-balls} therefore give
\begin{equation*}
    \omega_\eps
    \le\frac{C_{R,R}+C_{R,a}}{2\pi}\sqrt{\eps}
    \longrightarrow0\qquad{\rm{as}}\ \eps\to0.
\end{equation*}
The shell indicator is the outer-ball indicator minus the inner-ball
indicator. Its variation is at most $4$.
This proves part~(iii).

\subsection{Finite unions}

We first consider
\begin{equation*}
    S=\bigcup_{j=1}^{m}B(c_j,r_j).
\end{equation*}
The balls may overlap. Remove any duplicates.
The case $m=1$ was proved above, so assume $m\ge2$.
We choose a short boundary arc of one ball that stays a positive distance
from the others.

For a unit vector $u$, put $h_j(u)=\inpro{c_j}{u}+r_j$.
For distinct balls, the equality $h_i(u)=h_j(u)$ has empty interior on the
unit sphere. Choose $u_*$ outside these finitely many equality sets, and
relabel so that $h_1(u_*)>h_j(u_*)$ for all $j\ne1$.
Write $c_0=c_1$ and $R=r_1$. Choose a unit vector orthogonal to $u_*$ to form
the plane $E$, with $e_1=u_*$. By continuity, there are $\alpha\in(0,\pi)$
and $\gamma>0$ such that
\begin{equation}\label{eq:geom-exposed-gap}
    h_1(u_\theta)-h_j(u_\theta)\ge\gamma
    \qquad(-\alpha\le\theta\le\alpha,\ j\ne1).
\end{equation}
In particular, the point $c_0+Ru_\theta$ has distance at least $\gamma$
from every other closed ball. Indeed,
\begin{equation*}
    \inpro{c_0+Ru_\theta-c_j}{u_\theta}
    =r_j+h_1(u_\theta)-h_j(u_\theta)\ge r_j+\gamma.
\end{equation*}

Fix
\begin{equation*}
    \eps_0=\min\left\{\frac R{10},\frac\gamma4\right\},
    \qquad \d\nu(\theta)=\frac{\d\theta}{2\alpha}
    \quad(-\alpha\le\theta\le\alpha),
\end{equation*}
and define
\begin{equation}\label{eq:geom-union-windows}
    \begin{aligned}
        V_{\eps,\theta}&=S-c_0-(R-\eps)u_\theta,
        &W_{\eps,\theta}&=S-c_0-(R+\eps)u_\theta,\\
        Q_\eps&=B(0,\eps/4).
    \end{aligned}
\end{equation}
The first ball gives $Q_\eps\subset V_{\eps,\theta}$.
Its translate in $W_{\eps,\theta}$ misses $Q_\eps$.
For $x\in Q_\eps$, the point $x+c_0+(R+\eps)u_\theta$ is within
$5\eps/4<\gamma$ of $c_0+Ru_\theta$. By \eqref{eq:geom-exposed-gap},
it belongs to none of the other balls. Hence
$Q_\eps\subset V_{\eps,\theta}\setminus W_{\eps,\theta}$.
All windows are contained in the fixed ball
\begin{equation*}
    U=B(0,M+2R),\qquad
    M=\max_{1\le j\le m}\bigl(|c_j-c_0|+r_j\bigr).
\end{equation*}

Let $V_{\eps,\theta,j}$ and $W_{\eps,\theta,j}$ be the translates of the
$j$-th ball appearing in \eqref{eq:geom-union-windows}. We have
\begin{equation}\label{eq:geom-union-loss}
    W_{\eps,\theta}\setminus V_{\eps,\theta}
    \subset
    \bigcup_{j=1}^{m}
    \bigl(W_{\eps,\theta,j}\setminus V_{\eps,\theta,j}\bigr).
\end{equation}
Applying \eqref{eq:geom-outer-loss} with $a=r_j$ and $b=c_j-c_0$ yields
\begin{equation*}
    \omega_\eps
    \le\frac1{2\alpha}\sum_{j=1}^{m}C_{R,r_j}\sqrt{\eps}
    \longrightarrow0\qquad{\rm{as}}\ \eps\to0.
\end{equation*}
There is no restriction on the sizes of $r_j$ relative to $R$ in this estimate.
The variation of a union indicator is at most the sum of the variations of
the component indicators, so $J=2m$ suffices. This proves part~(iv).

We next allow holes, under the separation assumption in part~(v). Write
\begin{equation*}
    S=\bigcup_{j=1}^{m}S_j,\qquad
    S_j=B(c_j,b_j)\setminus B(c_j,a_j),\qquad 0\le a_j<b_j,
\end{equation*}
where $B(c_j,0)=\emptyset$, so $a_j=0$ gives the ball $B(c_j,b_j)$.
The case $m=1$ has already been proved. For $m\ge2$, set
\begin{equation*}
    \gamma=\min_{i\ne j}\operatorname{dist}
       (\overline{S_i},\overline{S_j})>0.
\end{equation*}
Relabel so that $b_1=R:=\max_j b_j$, and put $c_0=c_1$.
Every inner radius is strictly less than $R$, so
\begin{equation*}
    g:=\min_j(R-a_j)>0.
\end{equation*}
Use \eqref{eq:geom-union-windows} along a full great circle, now with
\begin{equation*}
    0<\eps<\eps_0:=\min\left\{\frac R{10},\frac g4,\frac\gamma4\right\},
    \qquad \d\nu(\theta)=\frac{\d\theta}{2\pi}.
\end{equation*}
For $x\in Q_\eps$, the inward translate places
$x+c_0+(R-\eps)u_\theta$ inside $S_1$.
The point $x+c_0+(R+\eps)u_\theta$ lies outside the outer ball of $S_1$
and within $5\eps/4<\gamma$ of $c_0+Ru_\theta\in\overline{S_1}$.
It therefore belongs to none of the other components.
This proves the common-set condition for the entire union.
We may again take $U=B(0,M+2R)$, with
$M=\max_j(|c_j-c_0|+b_j)$.

For each component, its loss is contained in its outer-ball loss together
with its inner-hole loss, as in \eqref{eq:geom-shell-loss}.
Combining this fact with \eqref{eq:geom-union-loss} gives
\begin{equation*}
    \omega_\eps
    \le\frac1{2\pi}\left(
        \sum_{j=1}^{m}C_{R,b_j}
        +\sum_{j:a_j>0}C_{R,a_j}\right)\sqrt{\eps}
    \longrightarrow0\qquad{\rm{as}}\ \eps\to0.
\end{equation*}
The reverse-loss estimate is applicable to every inner hole because
$a_j<R$ and $\eps<(R-a_j)/4$. The variation bound is
$J\le2m+2\#\{j:a_j>0\}\le4m$.
The abstract criterion proves part~(v).

In particular, finitely many concentric radial bands are covered by this
argument after overlapping or touching radial intervals are merged.
The remaining closed bands have positive gaps between them.

\subsection{Convex polytopes}

For part~(ii), we choose a straight path in the relative interior of one
facet. Write the polytope using one inequality for each facet
\begin{equation*}
    P=\bigcap_{i=0}^{m-1}\{x\in\R^d:\inpro{n_i}{x}\le a_i\},
    \qquad |n_i|=1,
\end{equation*}
and let $F_0$ be a facet whose supporting hyperplane is parallel to no other
facet hyperplane. Put $n=n_0$.

Following Schneider \cite{Sch14}, we write
$\operatorname{relint}F_0$ for the \emph{relative interior} of $F_0$,
namely its interior in its affine hull. Since $F_0$ is a facet, this affine
hull is its supporting hyperplane. For an edge of a polygon, the relative
interior is the edge with its two endpoints removed.

To explain the idea, we first consider the right triangle
$P=\{(x,y)\in\R^2:x\ge0,\ y\ge0,\ x+y\le1\}$,
shown in Figure~\ref{fig:triangle-windows}.
Take $F_0=[0,1]\times\{0\}$, $n=(0,-1)$ and
$z(t)=(s(t),0)$, where $s(t)=(1+t)/3$ for $0\le t\le1$.
For $0<\eps<1/20$, define
\begin{equation*}
    \begin{aligned}
        V_{\eps,t}&=P-z(t)+\eps n,
        &W_{\eps,t}&=P-z(t)-\eps n,\\
        Q_\eps&=B(0,\eps/2),
        &U&=B(0,2).
    \end{aligned}
\end{equation*}
All windows lie in $U$, and $Q_\eps\subset V_{\eps,t}\setminus W_{\eps,t}$
for every $t$. The loss region satisfies
\begin{equation*}
    L_{\eps,t}\subset
    \{(x,y):1-s(t)-\eps<x+y\le1-s(t)+\eps\}.
\end{equation*}
Since $s'(t)=1/3$, each fixed point lies in this strip for a set of times
of measure at most $6\eps$.
The inequalities defining $V_{\eps,t}$ are linear in $t$.
Thus the allowed times form an interval, and $J=2$.
If $P$ admits an exponential Riesz basis, condition~(i) also holds.
Theorem~\ref{thm:abstract-windows} then gives a contradiction.

\begin{figure}[H]
    \centering
    \begin{tikzpicture}[x=4.6cm,y=4.6cm,every node/.style={font=\small}]
        \definecolor{triangleblue}{RGB}{35,96,156}
        \definecolor{triangleorange}{RGB}{186,87,26}
        \definecolor{trianglegreen}{RGB}{25,120,62}
        \begin{scope}
            \node at (0.5,1.30) {(a) Path on the edge};
            \draw[black!70,line width=0.9pt] (0,0)--(1,0)--(0,1)--cycle;
            \node[anchor=south east] at (0,1) {$(0,1)$};
            \node[anchor=north east] at (0,0) {$(0,0)$};
            \node[anchor=west] at (1.015,0) {$(1,0)$};
            \node at (0.24,0.44) {$P$};
            \draw[black,line width=1.25pt] (0,0)--(1,0);
            \draw[black,line width=1.8pt,->,>=stealth]
                (1/3,0)--(2/3,0);
            \fill[triangleblue] (1/3,0) circle (0.012);
            \fill[triangleorange] (2/3,0) circle (0.012);
            \node[anchor=north,triangleblue] at (1/3,-0.025) {$z(0)$};
            \node[anchor=north,triangleorange] at (2/3,-0.025) {$z(1)$};
            \node[anchor=south] at (0.5,0.045) {$z(t)$};
            \node[anchor=north] at (0.93,-0.025) {$F_0$};
            \draw[black!65,->,>=stealth,line width=0.8pt]
                (0.5,-0.03)--(0.5,-0.31);
            \node[anchor=west] at (0.52,-0.27) {$n=(0,-1)$};
        \end{scope}
        \begin{scope}[shift={(2.05,0)}]
            \node at (0,1.30) {(b) The two windows};
            \foreach \s/\col in {0.333333/triangleblue,0.666667/triangleorange} {
                \fill[\col!28]
                    (-\s,0.96)--(-\s,1.04)--({1-\s},0.04)
                    --({0.92-\s},0.04)--cycle;
            }
            \foreach \s/\col in {0.333333/triangleblue,0.666667/triangleorange} {
                \draw[\col,line width=0.9pt,dash pattern=on 3pt off 2pt]
                    (-\s,0.04)--({1-\s},0.04)--(-\s,1.04)--cycle;
                \draw[\col,line width=1pt]
                    (-\s,-0.04)--({1-\s},-0.04)--(-\s,0.96)--cycle;
            }
            \draw[gray!70!black,densely dotted,->,>=stealth,line width=0.8pt]
                (-0.34,1.13)--(-0.66,1.13);
            \node[anchor=west,text=gray!70!black] at (-0.27,1.13)
                {$t$ increases};
            \filldraw[fill=trianglegreen!25,draw=trianglegreen,line width=0.8pt]
                (0,0) circle (0.02);
            \draw[trianglegreen,line width=0.5pt]
                (0.015,-0.015)--(0.20,-0.19);
            \node[anchor=west,trianglegreen] at (0.22,-0.19)
                {$Q_\eps$ (fixed)};
            \draw[triangleblue,line width=0.5pt]
                (0.16,0.507)--(0.45,0.69);
            \node[anchor=south,triangleblue,fill=white,inner sep=1pt]
                at (0.45,0.70) {$L_{\eps,0}$};
            \draw[triangleorange,line width=0.5pt]
                (-0.16,0.493)--(-0.51,0.68);
            \node[anchor=south,triangleorange,fill=white,inner sep=1pt]
                at (-0.51,0.69) {$L_{\eps,1}$};
            \node[anchor=west,triangleblue] at (-0.70,-0.17) {$t=0$};
            \node[anchor=west,triangleorange] at (-0.70,-0.28) {$t=1$};
            \draw[line width=1pt] (-0.22,-0.36)--(-0.08,-0.36);
            \node[anchor=west] at (-0.055,-0.36) {$V_{\eps,t}$};
            \draw[line width=0.9pt,dash pattern=on 3pt off 2pt]
                (0.32,-0.36)--(0.46,-0.36);
            \node[anchor=west] at (0.485,-0.36) {$W_{\eps,t}$};
        \end{scope}
    \end{tikzpicture}
    \caption{The right-triangle construction with $\eps=1/25$.
    The path $z(t)$ runs from $(1/3,0)$ to $(2/3,0)$ inside
    $\operatorname{relint}F_0$. The right panel shows
    $V_{\eps,t}=P-z(t)+\eps n$ with solid boundaries and
    $W_{\eps,t}=P-z(t)-\eps n$ with dashed boundaries at $t=0$ (blue)
    and $t=1$ (orange). The green ball is the same $Q_\eps$ at both times.
    The shaded regions are $L_{\eps,t}=W_{\eps,t}\setminus V_{\eps,t}$.
    They move to the left as $t$ increases and lie in strips parallel
    to the hypotenuse.}
    \label{fig:triangle-windows}
\end{figure}
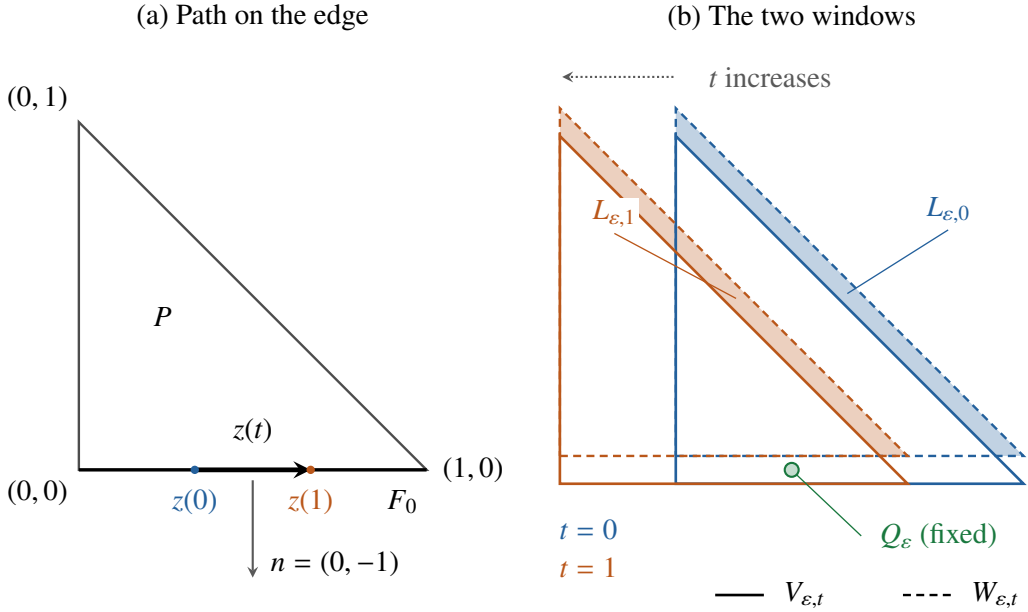

We now return to the general polytope $P$ and the facet $F_0$ fixed above.
For every $i\ne0$, the linear functional
$\tau\mapsto\inpro{n_i}{\tau}$ is nonzero on $n^\perp$.
Otherwise $n_i$ would be parallel to $n$. Finitely many proper linear
subspaces cannot cover $n^\perp$, so there is a unit vector $\tau$ such that
\begin{equation*}
    \inpro{n}{\tau}=0,\qquad
    \inpro{n_i}{\tau}\ne0\quad(i\ne0).
\end{equation*}
Choose $z_0\in\operatorname{relint}F_0$ and $\rho>0$ small enough that
\begin{equation*}
    z(t)=z_0+(2t-1)\rho\tau\in\operatorname{relint}F_0
    \qquad(0\le t\le1).
\end{equation*}
The segment is compact, so
\begin{equation}\label{eq:geom-facet-margin}
    \eta:=\min_{\substack{i\ne0\\0\le t\le1}}
       \bigl(a_i-\inpro{n_i}{z(t)}\bigr)>0.
\end{equation}

Set $\eps_0=\eta/4$. For $0<\eps<\eps_0$, define
\begin{equation}\label{eq:geom-polytope-windows}
    \begin{aligned}
        V_{\eps,t}&=P-z(t)+\eps n,
        &W_{\eps,t}&=P-z(t)-\eps n,\\
        Q_\eps&=B(0,\eps/2),
        &U&=B(0,\operatorname{diam}P+\eps_0).
    \end{aligned}
\end{equation}
A point in either window has the form $p-z(t)\pm\eps n$, with $p,z(t)\in P$.
Its norm is less than $\operatorname{diam}P+\eps_0$, so it lies in $U$. For $x\in Q_\eps$, the point $z(t)-\eps n+x$ satisfies
\begin{equation*}
    \begin{aligned}
        \inpro{n}{z(t)-\eps n+x}&<a_0-\eps/2,\\
        \inpro{n_i}{z(t)-\eps n+x}
        &<a_i-\eta+3\eps/2<a_i\quad(i\ne0).
    \end{aligned}
\end{equation*}
Thus $x\in V_{\eps,t}$. On the other hand,
$\inpro{n}{z(t)+\eps n+x}>a_0+\eps/2$, so $x\notin W_{\eps,t}$.
This proves $Q_\eps\subset V_{\eps,t}\setminus W_{\eps,t}$.

It remains to estimate the loss. For fixed $\eps$, put
\begin{equation*}
    \beta_i(t)=a_i-\inpro{n_i}{z(t)}-\eps\inpro{n_i}{n}.
\end{equation*}
The $i$-th inequality for $W_{\eps,t}$ has right-hand side $\beta_i(t)$,
whereas that for $V_{\eps,t}$ has right-hand side
$\beta_i(t)+2\eps\inpro{n_i}{n}$.
Only indices with $\inpro{n_i}{n}<0$ can therefore contribute to the loss.
More precisely,
\begin{equation}\label{eq:geom-polytope-slabs}
    L_{\eps,t}\subset
    \bigcup_{i:\inpro{n_i}{n}<0}
    \left\{x:
       \beta_i(t)-2\eps|\inpro{n_i}{n}|
       <\inpro{n_i}{x}\le\beta_i(t)\right\}.
\end{equation}
Since $\beta_i'(t)=-2\rho\inpro{n_i}{\tau}\ne0$ for these indices,
a fixed point belongs to the $i$-th slab for a set of parameters of length
at most $\eps|\inpro{n_i}{n}|/(\rho|\inpro{n_i}{\tau}|)$.
With Lebesgue measure on $[0,1]$, we obtain
\begin{equation}\label{eq:geom-polytope-overlap}
    \omega_\eps
    \le\frac\eps\rho
       \sum_{i:\inpro{n_i}{n}<0}
       \frac{|\inpro{n_i}{n}|}{|\inpro{n_i}{\tau}|}
    \longrightarrow0\qquad{\rm{as}}\ \eps\to0.
\end{equation}
For each fixed $x$, membership in $V_{\eps,t}$ is an intersection of
linear inequalities in $t$. Its parameter set is an interval, possibly empty
or a single point. The variation bound is therefore $J=2$.
This proves part~(ii).

Any two distinct facets of a nondegenerate $d$-simplex meet along a
$(d-2)$-dimensional face, so their supporting hyperplanes cannot be parallel.
This includes all nondegenerate triangles.
For a convex polygon, each unoriented edge direction occurs at most twice
when edges are taken to be maximal boundary segments. If the number of edges
is odd, at least one edge has no parallel partner.
These observations justify the corresponding examples in the Introduction.

\subsection{Products and affine images}

We use the same windows to prove the product and affine extensions.

Let $K\subset\R^k$ be Lebesgue measurable with $0<|K|<\infty$.
Replacing $K$ by a Borel representative if necessary, we may assume that it
is Borel. Given the windows for $S$, take
\begin{equation*}
    \begin{aligned}
    \widetilde V_{\eps,t}&=V_{\eps,t}\times K,&
    \widetilde W_{\eps,t}&=W_{\eps,t}\times K,\\
    \widetilde Q_\eps&=Q_\eps\times K,&
    \widetilde U&=U\times K.
    \end{aligned}
\end{equation*}
Here $\widetilde U$ has finite measure, although it may be unbounded.
Assume that $E(\Lambda)$ is a Riesz basis on $S\times K$ with bounds $A,B$.
Since $S$ has nonempty interior and $U$ is bounded, choose a fixed finite cover
$U\subset\bigcup_{j=1}^{m}(S+v_j)$. Then
\begin{equation*}
    \widetilde U\subset\bigcup_{j=1}^{m}
    \bigl((S\times K)+(v_j,0)\bigr).
\end{equation*}
Translations preserve the Riesz bounds on the whole product. Thus,
for every finitely supported coefficient vector $c$,
\begin{equation*}
    \int_{\widetilde U}
    \left|\sum_{\lambda\in\Lambda}c_\lambda
    e^{2\pi i\inpro{\lambda}{(x,y)}}\right|^2\d x\,\d y
    \le mB\norm{c}_{\ell^2(\Lambda)}^2.
\end{equation*}
This defines a common bounded synthesis operator into $L^2(\widetilde U)$.
Both windows are translates of $S\times K$, so condition~(i) holds.
No assertion about bases on the separate factors is needed.

The common set has measure $|\widetilde Q_\eps|=|Q_\eps||K|>0$,
and condition~(ii) follows from the corresponding inclusion for $S$.
Moreover,
\begin{equation*}
    \widetilde W_{\eps,t}\setminus\widetilde V_{\eps,t}
    =L_{\eps,t}\times K,\qquad
    \1_{\widetilde V_{\eps,t}}(x,y)
    =\1_{V_{\eps,t}}(x)\1_K(y).
\end{equation*}
The overlap and variation bounds are unchanged. A common exceptional null
set $N$ for the first factor becomes $N\times K$, which is still null.
Joint measurability is also preserved. Thus all four conditions of
Theorem~\ref{thm:abstract-windows} hold, giving the required contradiction.

Now let $\Phi(x)=Lx+v$ be invertible and affine. Replace $U,Q_\eps$ and both
windows by their images under $\Phi$. If $V_{\eps,t}=S+a_{\eps,t}$, then
\begin{equation*}
    \Phi(V_{\eps,t})=\Phi(S)+La_{\eps,t},
\end{equation*}
and the same identity holds for the second window.
Thus the new windows are translates of $\Phi(S)$.
Their common set has measure
$|\Phi(Q_\eps)|=|\det L|\,|Q_\eps|>0$.
The identity
\begin{equation*}
    \1_{\Phi(E_0)}(\Phi(x))=\1_{E_0}(x)
\end{equation*}
for every measurable $E_0$ preserves both the pointwise variation and the
parameter overlap bounds. Invertible affine maps also preserve null sets.
Applying $\Phi$ to the finite cover gives a finite cover by translates
of $\Phi(S)$. Thus the abstract criterion applies to $\Phi(S)$.
For products, the same proof applies because $|\Phi(\widetilde U)|<\infty$.

This completes the proof of Theorem~\ref{thm:main}.

\section{Weaker assumptions}\label{sec:weaker-assumptions}

In this section, we weaken conditions~(iii) and~(iv) of
Theorem~\ref{thm:abstract-windows}. The preceding proofs do not use this section.

We keep conditions~(i) and~(ii), which give a positive lower bound on the loss
energy. To obtain a contradiction, we need an upper bound that tends to zero.
In the original proof, condition~(iii) controls the loss for each fixed
function. Condition~(iv) lets us approximate the varying coefficient family
by finitely many vectors, with a number bounded independently of $\eps$
at each fixed accuracy.

We first replace these two conditions by a joint estimate. It combines
the weighted loss of a finite-dimensional subspace with the mean square
error in approximating the coefficients by that subspace.
We then keep condition~(iii) and give a replacement for condition~(iv)
that involves only the windows. The last subsection applies this condition
to countable unions of balls.

\subsection{Weighted mean approximation}

Retain the setting and conditions~(i) and~(ii) of
Theorem~\ref{thm:abstract-windows}. Choose the common function
$f_\eps=|Q_\eps|^{-1/2}\1_{Q_\eps}$, and put
\begin{equation*}
    a_\eps=T^*f_\eps,\qquad c_{\eps,t}=X_{\eps,t}^{-1}a_\eps,
    \qquad q_\eps(x)=\int_I\1_{L_{\eps,t}}(x)\,\d\nu(t).
\end{equation*}
The proof in Section~\ref{sec:main-proof} gives
$\norm{c_{\eps,t}}\le A^{-1/2}$ and
$\inpro{Z_{\eps,t}c_{\eps,t}}{c_{\eps,t}}\ge A/B$, where
$Z_{\eps,t}=T^*P_{L_{\eps,t}}T$ and $0\le Z_{\eps,t}\le BI$.
Neither conclusion uses condition~(iii) or~(iv).

\begin{proposition}\label{prop:weighted-mean}
For each $\eps$, let $P_\eps$ be a finite-rank orthogonal projection in $H$,
fixed as $t$ varies.
Choose an orthonormal basis $e_{\eps,1},\ldots,e_{\eps,k_\eps}$ of its range,
and define
\begin{equation}\label{eq:weighted-quantities}
    k_{P_\eps}(x)=\sum_{j=1}^{k_\eps}|Te_{\eps,j}(x)|^2,\quad
    \ell_\eps=\int_U q_\eps(x)k_{P_\eps}(x)\,\d x,\quad
    d_\eps^2=\int_I\norm{(I-P_\eps)c_{\eps,t}}^2\d\nu(t).
\end{equation}
Then
\begin{equation}\label{eq:weighted-mean-bound}
    \sqrt{\frac AB}
    \le\left(\int_I\inpro{Z_{\eps,t}c_{\eps,t}}{c_{\eps,t}}
       \,\d\nu(t)\right)^{1/2}
    \le\frac{\sqrt{\ell_\eps}}{\sqrt A}+\sqrt B\,d_\eps.
\end{equation}
In particular, conditions~(iii) and~(iv) may be replaced by the existence
of such projections with $\ell_\eps\to0$ and $d_\eps\to0$ as $\eps\to0$.
\end{proposition}
\begin{proof}
Write $c_{\eps,t}=P_\eps c_{\eps,t}+(I-P_\eps)c_{\eps,t}$.
By Cauchy--Schwarz,
\begin{equation*}
    |TP_\eps c_{\eps,t}(x)|^2
    \le\norm{P_\eps c_{\eps,t}}^2 k_{P_\eps}(x)
    \le A^{-1}k_{P_\eps}(x).
\end{equation*}
By Tonelli's theorem, the squared norm of
$\1_{L_{\eps,t}}TP_\eps c_{\eps,t}$ in $L^2(I\times U,\nu\times\d x)$
is at most $\ell_\eps/A$. The squared norm of the remaining part is at most
$Bd_\eps^2$, since $Z_{\eps,t}\le BI$.
The triangle inequality in this space proves the upper bound in
\eqref{eq:weighted-mean-bound}. Integrating the positive lower bound on
the loss energy proves the other inequality.
\end{proof}

Here $d_\eps$ measures the mean approximation error, while $\ell_\eps$
measures the weighted loss of the same subspace. The projection may depend
on $\eps$ and on the chosen coefficient family, but not on $t$.
The estimate only tests $q_\eps$ against this subspace and does not require
a uniform bound tending to zero in $x$. In comparison, condition~(iii) gives
\begin{equation*}
    \ell_\eps\le Ck_\eps\omega_\eps.
\end{equation*}
To recover the original argument, fix an accuracy $\delta$. Condition~(iv) and
Lemma~\ref{lem:main-inverse-cover} provide a subspace of dimension at most
$N_\delta$, where $N_\delta$ is independent of $\eps$, with $d_\eps\le\delta$.
First letting $\eps\to0$ and then $\delta\to0$ in
\eqref{eq:weighted-mean-bound} recovers the original contradiction.
A diagonal choice of these subspaces also gives both zero limits in the
proposition. The new estimate allows the dimensions to grow as $\eps\to0$,
provided that the weighted loss still tends to zero.

One can express the weighted condition using a probability measure.
If $s_\eps=\int_U k_{P_\eps}\,\d x>0$, set
$\d\mu_\eps=k_{P_\eps}\,\d x/s_\eps$. Then
\begin{equation*}
    \ell_\eps=s_\eps\int_U q_\eps\,\d\mu_\eps,
    \qquad Ak_\eps\le s_\eps\le Ck_\eps.
\end{equation*}
The factor $s_\eps$ cannot be omitted when the dimension increases.
Decay of $\int_Uq_\eps\,\d\mu_\eps$ alone is therefore not enough.
We must control its product with $s_\eps$, together with the approximation
error for the same subspace.

\subsection{Mean approximation of the windows}

The preceding estimate involves the coefficient family. To apply it using
only information about the windows, we retain conditions~(i)--(iii) and
replace condition~(iv) by the following condition.
For every $\eta>0$ there is an integer $N(\eta)$ such that, for every $\eps$,
the interval $I$ has a measurable partition $E_1,\ldots,E_n$, with
$n\le N(\eta)$, satisfying
\begin{equation}\label{eq:mean-window-partition}
    \operatorname*{ess\,sup}_{x\in U}
    \sum_{j=1}^{n}\int_{E_j}
    |\1_{V_{\eps,t}}(x)-p_{\eps,j}(x)|^2\d\nu(t)\le\eta^2,
    \qquad
    p_{\eps,j}(x)=\frac1{\nu(E_j)}
      \int_{E_j}\1_{V_{\eps,t}}(x)\,\d\nu(t).
\end{equation}
Blocks of zero measure are discarded. The partition may depend on $\eta$
and $\eps$, but it is independent of $x$.
Within each block, $p_{\eps,j}(x)$ is the average of the first-window indicator.
The condition bounds the mean square error from these averages.
It does not count how many times a point enters a window.

\begin{proposition}\label{prop:mean-window-criterion}
The conclusion of Theorem~\ref{thm:abstract-windows} remains valid when
condition~(iv) is replaced by \eqref{eq:mean-window-partition}.
\end{proposition}
\begin{proof}
We first approximate $X_{\eps,t}v$ for fixed $v$. Put
$\overline X_{\eps,j}=T^*M_{p_{\eps,j}}T$, where $M_p$ is multiplication by $p$.
These operators are the averages of $X_{\eps,t}$ on the corresponding blocks,
so $AI\le\overline X_{\eps,j}\le BI$.
For every fixed $v\in H$, Tonelli's theorem and
\eqref{eq:mean-window-partition} give
\begin{equation}\label{eq:mean-direct-estimate}
    \sum_{j=1}^{n}\int_{E_j}
    \norm{(X_{\eps,t}-\overline X_{\eps,j})v}^2\d\nu(t)
    \le C\int_U|Tv(x)|^2
       \sum_{j=1}^{n}\int_{E_j}|\1_{V_{\eps,t}}(x)-p_{\eps,j}(x)|^2
       \d\nu(t)\,\d x
    \le C^2\eta^2\norm v^2.
\end{equation}
Thus the family $\{X_{\eps,t}v:t\in I\}$ has a mean approximation taking
finitely many values. Their number is bounded independently of $\eps$
and uniformly for $v$ in a bounded ball.

We next pass from these fixed vectors to the coefficients involving $X_{\eps,t}^{-1}$.
The vectors in the iteration depend on $t$, so this step needs a separate
argument. Omit $\eps$, and set
$q=1-A/B$, $R_t=I-X_t/B$ and $\rho=\sqrt B/A$.
If $A=B$, then $X_t=AI$ and the coefficient family is constant.
Otherwise $0<q<1$. For $a=a_\eps$, the iterates
$y_{0,t}=0$ and $y_{m+1,t}=a/B+R_ty_{m,t}$ satisfy
\begin{equation*}
    \norm{y_{m,t}}\le\rho,\qquad
    \norm{X_t^{-1}a-y_{m,t}}\le\rho q^m.
\end{equation*}
Given $h>0$ and fixed $z$ with $\norm z\le\rho$,
\eqref{eq:mean-direct-estimate} gives a simple function $b_z(t)$ with
$\norm{R_tz-b_z(t)}_{L^2(\nu)}\le h$.
The number of values is bounded in terms of $h$ and the fixed constants.
We take these values to be $(I-\overline X_{\eps,j}/B)z$.
Their norms are at most $q\rho$.

Fix $\delta>0$. Starting with $z_{0,t}=0$, suppose that $z_{m,t}$ has
$L_m$ values, all of norm at most $\rho$. For each of these fixed values $z$,
choose $b_z$ with accuracy
\begin{equation*}
    h_m=\frac{(1-q)\delta}{2\sqrt{L_m}},
    \qquad z_{m+1,t}=a/B+b_{z_{m,t}}(t).
\end{equation*}
Each $z_{m+1,t}$ still has finitely many values of norm at most $\rho$.
At each fixed step, their number has a bound determined by $\delta$, the
fixed constants, and the bound from the preceding step. It is therefore
independent of $\eps$.
If $e_m=\norm{y_m-z_m}_{L^2(\nu)}$, then
\begin{equation*}
    e_{m+1}\le qe_m+\sqrt{L_m}\,h_m
    =qe_m+(1-q)\delta/2.
\end{equation*}
For each value of $z_m$, we bound the squared error on its level set by
the squared error for that fixed value over all of $I$.
Summing over the $L_m$ values gives the factor $\sqrt{L_m}$ above.
Thus $e_m\le\delta/2$.
Choosing $m$, independently of $\eps$, with $\rho q^m\le\delta/2$ gives a mean $\delta$-approximation
to $c_{\eps,t}$ by at most $L_\delta$ values, where $L_\delta$ is independent
of $\eps$.

Let $P_\eps$ project onto their span. Then $d_\eps\le\delta$ and,
by condition~(iii), $\ell_\eps\le CL_\delta\omega_\eps$.
Proposition~\ref{prop:weighted-mean}, first with $\eps\to0$ and then with
$\delta\to0$, gives a contradiction.
\end{proof}

Thus mean approximation of the windows gives the coefficient approximation
needed in Proposition~\ref{prop:weighted-mean}, while condition~(iii) controls
the weighted loss.

For normalized Lebesgue measure on an interval, condition~(iv) implies
\eqref{eq:mean-window-partition}. Divide the interval into $m$ equal blocks.
A function taking values zero and one with at most $J$ changes is constant
on all but at most $J$ blocks, up to endpoints. Its mean square error
from the block averages is therefore at most $J/(4m)$.
The example below shows that the new condition can hold for a family of
windows with no bound on its variation independent of $\eps$.

\subsection{Countable unions of balls}

We now apply this condition to countable unions of balls.
We first handle a finite number of balls using the estimates from
Section~\ref{sec:3}. The remaining balls will have a small total contribution
to the parameter average. This lets us control the mean error without
counting every entry and exit.

\begin{proposition}\label{prop:countable-balls}
Let $d\ge2$, $R>0$, and
\begin{equation}\label{eq:countable-ball-set}
    S=B(0,R)\cup\bigcup_{j=1}^{\infty}B(b_j,r_j),
    \qquad \sup_j|b_j|<\infty,\qquad
    0<r_j<R/10,\qquad \sum_{j=1}^{\infty}r_j<\infty.
\end{equation}
Suppose that, for some closed interval $I_0$ of positive length at most
$2\pi$ and some orthonormal $e_1,e_2$, the arc
$\Gamma=\{Ru_\theta:\theta\in I_0\}$ satisfies
\begin{equation}\label{eq:countable-exposed-arc}
    \gamma:=\operatorname{dist}\left(
    \Gamma,\overline{\bigcup_{j=1}^{\infty}B(b_j,r_j)}\right)>0,
    \qquad u_\theta=e_1\cos\theta+e_2\sin\theta.
\end{equation}
Then $L^2(S)$ admits no Riesz basis of exponentials.
The same conclusion holds for its invertible affine images, and for
$S\times K$ and their invertible affine images whenever
$K\subset\R^k$ is measurable and $0<|K|<\infty$.
\end{proposition}
\begin{proof}
Choose $M$ such that $S\subset B(0,M)$, and set
$\eps_0=\min\{R/10,\gamma/4\}$. For $0<\eps<\eps_0$, take
\begin{equation}\label{eq:countable-windows}
    \begin{aligned}
    V_{\eps,\theta}&=S-(R-\eps)u_\theta,&
    W_{\eps,\theta}&=S-(R+\eps)u_\theta,\\
    Q_\eps&=B(0,\eps/4),& U&=B(0,M+2R),
    \end{aligned}
    \qquad \d\nu(\theta)=\frac{\d\theta}{|I_0|}.
\end{equation}
These windows are jointly measurable and lie in $U$.
Assume that $S$ admits a Riesz basis of exponentials.
The synthesis operator from Section~\ref{sec:main-proof} gives condition~(i).
The estimates for $B(0,R)$ give
$Q_\eps\subset B(-(R-\eps)u_\theta,R)$ and exclude $Q_\eps$ from
$B(-(R+\eps)u_\theta,R)$.
For $x\in Q_\eps$, the distance from $x+(R+\eps)u_\theta$ to
$Ru_\theta$ is less than $5\eps/4<\gamma$.
Thus it also lies outside every small ball. This proves condition~(ii).

We first bound the time spent in a small ball. If $0<r<R/10$ and
$9R/10\le\rho\le11R/10$, then, uniformly in $y\in\R^d$,
\begin{equation}\label{eq:small-ball-visitation}
    |\{\theta\in I_0:|\rho u_\theta-y|<r\}|
    \le 2\arcsin(r/\rho)\le\frac{10\pi r}{9R}.
\end{equation}
Indeed, put $p=|\operatorname{proj}_{\operatorname{span}\{e_1,e_2\}}y|$.
If $p=0$, the set is empty. Otherwise write the planar projection as $p u_\phi$.
The allowed directions satisfy
\begin{equation*}
    \cos(\theta-\phi)>
    \frac{\rho^2+|y|^2-r^2}{2\rho p}
    \ge\frac{\rho^2+p^2-r^2}{2\rho p}
    \ge\sqrt{1-r^2/\rho^2}.
\end{equation*}
This gives the first inequality in \eqref{eq:small-ball-visitation}.
The second follows from $\arcsin s\le\pi s/2$ for $0\le s\le1$.

Let $S_N=B(0,R)\cup\bigcup_{j=1}^{N}B(b_j,r_j)$, and let
$V^N_{\eps,\theta},W^N_{\eps,\theta}$ denote its corresponding translates.
Applying \eqref{eq:small-ball-visitation} to each remaining ball and summing gives
\begin{equation}\label{eq:countable-tail}
    \sup_{x\in U}\int_{I_0}
    \1_{V_{\eps,\theta}\setminus V^N_{\eps,\theta}}(x)\,\d\nu(\theta)
    \le\tau_N:=\frac{10\pi}{9R|I_0|}\sum_{j>N}r_j\longrightarrow0.
\end{equation}
The same bound holds with $W$ in place of $V$, uniformly in $\eps$.
This one estimate controls both the loss and the mean approximation.
For the loss, use
\begin{equation*}
    W_{\eps,\theta}\setminus V_{\eps,\theta}
    \subset (W^N_{\eps,\theta}\setminus V^N_{\eps,\theta})
       \cup(W_{\eps,\theta}\setminus W^N_{\eps,\theta}),
\end{equation*}
Lemma~\ref{lem:geom-moving-balls} gives
$\omega_\eps\le C_N\sqrt\eps+\tau_N$, with $C_N$ independent of $\eps$.
First letting $\eps\to0$ and then $N\to\infty$ proves condition~(iii).

For the mean approximation, divide $I_0$ into
$m$ equal intervals and let $\mathcal A_m$ take the average on each interval.
For fixed $x$, write
$g(\theta)=\1_{V_{\eps,\theta}}(x)$ and
$g_N(\theta)=\1_{V^N_{\eps,\theta}}(x)$.
The finite union has variation at most $2(N+1)$, so
\begin{equation}\label{eq:countable-mean-approximation}
    \norm{g-\mathcal A_m g}_{L^2(\nu)}
    \le\norm{g_N-\mathcal A_m g_N}_{L^2(\nu)}
       +\norm{g-g_N}_{L^2(\nu)}
    \le\sqrt{\frac{N+1}{2m}}+\sqrt{\tau_N}.
\end{equation}
Here $I-\mathcal A_m$ is an orthogonal projection, and the second term is
bounded by \eqref{eq:countable-tail}. Given $\eta>0$, first choose $N$ and
then $m$ so that the right-hand side is at most $\eta$.
Both choices are independent of $\eps$ and $x$.
Proposition~\ref{prop:mean-window-criterion} now gives the contradiction.

The product and affine arguments from Section~\ref{sec:3} apply without change.
The product windows have indicators multiplied by $\1_K(y)$,
so \eqref{eq:mean-window-partition} also holds.
\end{proof}

The next example compares the two conditions for one specified family of
windows. The small disks lie along a circular arc. As $\eps$ decreases,
the point $x_0$ used below enters arbitrarily many of their translates.

\begin{example}\label{ex:countable-disks}
In $\R^2$, take $R=1$, $I_0=[0,1]$, and
\begin{equation}\label{eq:explicit-countable-disks}
    x_0=(4,0),\qquad \theta_j=2^{-j},\qquad r_j=2^{-4j-20},
    \qquad
    S=B(0,1)\cup\bigcup_{j=1}^{\infty}B(x_0+u_{\theta_j},r_j).
\end{equation}
Then $L^2(S)$ admits no Riesz basis of exponentials.
The windows \eqref{eq:countable-windows} satisfy
\eqref{eq:mean-window-partition} but not condition~(iv).
\end{example}
\begin{proof}
The small disks have pairwise disjoint closures and are separated from the
unit disk. Indeed, for $j<k$, the distance between their centres is at least
$\theta_j/4$, while $r_j+r_k\le3\theta_j/32$.
Their centres converge to $(5,0)$ and their radii have finite sum.
They stay a positive distance from the whole unit circle.
Thus Proposition~\ref{prop:countable-balls} applies.
This bounded open set has infinitely many connected components, so it is
not one of the finite unions in Theorem~\ref{thm:main}.

We show that the windows \eqref{eq:countable-windows}, with $I_0=[0,1]$,
have no variation bound independent of $\eps$.
Put $g_j=3\cdot2^{-(j+2)}$ for $j\ge0$.
Given $M\ge1$, choose $0<\eps<r_M/8$ and
$x\in B(x_0,r_M/8)$. For $1\le j\le M$,
\begin{equation*}
    |x+(1-\eps)u_{\theta_j}-(x_0+u_{\theta_j})|
    \le |x-x_0|+\eps<r_M/4<r_j.
\end{equation*}
Hence $x\in V_{\eps,\theta_j}$. At $g_j$, where $0\le j\le M$,
write $\Delta=|g_j-\theta_k|$ for any $k\ge1$.
The definitions of $g_j$ and $\theta_k$ give
$\Delta\ge\theta_k/4$ and $\Delta\ge2^{-(M+2)}$.
Since $\Delta<1$, we have
$|u_{g_j}-u_{\theta_k}|\ge\Delta/2$ and
$r_k\le\theta_k/16\le\Delta/4$. It follows that
\begin{equation*}
    |x+(1-\eps)u_{g_j}-(x_0+u_{\theta_k})|-r_k
    \ge\Delta/4-|x-x_0|-\eps
    >2^{-(M+4)}-r_M/4>0.
\end{equation*}
Also $|x+(1-\eps)u_{g_j}|\ge4-r_M/8-1>1$, so the unit disk does not
contribute at these times. Thus $x\notin V_{\eps,g_j}$ for $0\le j\le M$.
Along the increasing sequence
\begin{equation*}
    g_M<\theta_M<g_{M-1}<\theta_{M-1}<\cdots<g_1<\theta_1<g_0,
\end{equation*}
the indicator of $V_{\eps,\theta}$ alternates between zero and one.
The sum in \eqref{eq:main-turnover} is therefore $2M$.
This holds throughout the positive-measure disk
$B(x_0,r_M/8)$. Changing representatives on null sets cannot remove this failure.

Since $\sum_j r_j<\infty$, \eqref{eq:countable-mean-approximation} holds
uniformly in $\eps$. Thus the mean condition is strictly weaker than
condition~(iv) for these windows.
\end{proof}

\begin{remark}\label{rem:countable-original-criterion}
Example~\ref{ex:countable-disks} shows that the mean condition is weaker than
condition~(iv) for the specified windows on $[0,1]$.
The set itself can still be treated by Theorem~\ref{thm:abstract-windows}.
Indeed, restrict the parameter interval to $I_1=[3/4,1]$ and take
$\d\nu_1=4\,\d\theta$ there. For $j<k$ and $\theta,\phi\in I_1$,
projection onto $u_{(\theta+\phi)/2}$ gives
\begin{equation*}
    |u_{\theta_j}-u_{\theta_k}-(1-\eps)(u_\theta-u_\phi)|
    \ge\sin(3/8)|u_{\theta_j}-u_{\theta_k}|>r_j+r_k.
\end{equation*}
Thus, for each fixed $\eps$, different small disks cover disjoint regions
as the windows move over $I_1$. Each point meets at most one small disk.
Including the unit disk gives the variation bound $J=4$.
Conditions~(i) and~(ii) are unchanged, and the new overlap bound is at most
$4\omega_\eps\to0$. A linear change of parameter from $I_1$ to $[0,1]$
then allows us to apply Theorem~\ref{thm:abstract-windows}.
\end{remark}

\section*{Acknowledgement}
The author sincerely thanks Professor Jingwei Guo for helpful discussions.
The author is supported by the National
Natural Science Foundation of China No.~12341102.

\begingroup
\def\L{{\fontencoding{T1}\selectfont\symbol{138}}}

\endgroup

\end{document}